\documentclass[11pt]{amsart}
\usepackage[usenames,dvipsnames]{pstricks}
\usepackage[mathscr]{eucal}
\usepackage{amsfonts,amsmath,amssymb,amsthm,amscd,amsxtra}
\usepackage{enumerate,verbatim}
\usepackage[all,2cell,ps]{xy}
\usepackage[notcite]{}
\usepackage[pagebackref]{hyperref}
\usepackage{calc,color}

\theoremstyle{plain} 

\newtheorem{thm}{Theorem}[section]
\newtheorem{dfn}[thm]{Definition}

\newtheorem{eg}[thm]{Example}

\newtheorem{rmk}[thm]{Remark}
\newtheorem{prop}[thm]{Proposition}

\newtheorem{cor}[thm]{Corollary}
\newtheorem{lem}[thm]{Lemma}

\numberwithin{equation}{section}

\newcommand{\fm}{\mathfrak{m}}

\newcommand{\fp}{\mathfrak{p}}

\newcommand{\fa}{\mathfrak{a}}
\newcommand{\fb}{\mathfrak{b}}
\newcommand{\fc}{\mathfrak{c}}

\newcommand{\Ic}{\mathcal{I}_C}

\DeclareMathOperator{\ann}{ann} \DeclareMathOperator{\Ass}{Ass}
 
\DeclareMathOperator{\E}{E} \DeclareMathOperator{\hh}{H}
 
 \DeclareMathOperator{\X}{X}

\def\ac{\operatorname{\mathcal{A}_{\it C}}}
\def\bc{\operatorname{\mathcal{B}_{\it C}}}
\def\gd{\operatorname{\mathsf{G-dim}}}

\def\gkd{\operatorname{\mathsf{G}_{\it C}\mathsf{-dim}}}
\def\gomegad{\operatorname{\mathsf{G}_{\it {\omega_R}}\mathsf{-dim}}}
\def\gkkd{\operatorname{\mathsf{G}_{\it K}\mathsf{-dim}}}
\def\gc{\operatorname{\mathsf{G}_{\it C}}}
\def\gcp{\operatorname{\mathsf{G}_{\it C_{\fp}}}}
\def\gcpd{\operatorname{\mathsf{G}_{\it C_p}\mathsf{-dim_{R_p}}}}

\def\rgr{\operatorname{\mathsf{r.grade}}}
\def\rgric{\operatorname{\mathsf{r.grade}_{\mathcal{I}_C}}}
\def\gr{\operatorname{\mathsf{grade}}}

\def\Tr{\mathsf{Tr}}
\def\trk{\mathsf{Tr}_{C}}
\def\tro{\mathsf{Tr}_{\omega_R}}
\def\trc{\mathcal{T}^C}

\def\ng{\mathsf{NG}_{C}}

\DeclareMathOperator{\coker}{Coker}

\def\depth{\operatorname{\mathsf{depth}}}
 
\def\Ext{\operatorname{\mathsf{Ext}}}

\def\Hom{\operatorname{\mathsf{Hom}}}

\DeclareMathOperator{\id}{id} \DeclareMathOperator{\im}{im}

\DeclareMathOperator{\Supp}{Supp} \DeclareMathOperator{\Spec}{Spec}

\def\Tor{\operatorname{\mathsf{Tor}}}

\def\urltilda{\kern -.15em\lower .7ex\hbox{\~{}}\kern .04em}
\def\urldot{\kern -.10em.\kern -.10em}\def\urlhttp{http\kern -.10em\lower -.1ex
\hbox{:}\kern -.12em\lower 0ex\hbox{/}\kern -.18em\lower
0ex\hbox{/}}

\begin{document}
\baselineskip=15pt

\title[Linkage of finite $\gc$-dimension modules]
 {Linkage of finite $\gc$-dimension modules}

\bibliographystyle{amsplain}

\author[A. Sadeghi]{Arash Sadeghi}\

\address{ School of Mathematics, Institute for Research in Fundamental Sciences (IPM), P.O. Box: 19395-5746, Tehran, Iran }
\email{sadeghiarash61@gmail.com}

\maketitle
\begin{abstract}
Let $R$ be a semiperfect commutative Noetherian ring and $C$ a semidualizing $R$--module. We study the theory of linkage for modules of finite $\gc$-dimension. For a horizontally linked $R$--module $M$ of finite $\gc$-dimension, the connection of the Serre condition $(S_n)$ with the vanishing of certain relative cohomology modules of its linked module is discussed.
\end{abstract}
\section{Introduction}
The theory of linkage of algebraic varieties was introduced by
C. Peskine and L. Szpiro \cite{PS}. Recall that two ideals $\fa$ and $\fb$
in a Gorenstein local ring $R$ are said to be linked by a Gorenstein ideal $\fc$ if $\fc\subseteq\fa\cap\fb$  such that
$\fa=(\fc):\fb$ and $\fb=(\fc):\fa$. The first main theorem in
the theory of linkage was due to C. Peskine and L. Szpiro. They
proved that over a Gorenstein local ring $R$ with linked ideals
$\fa$ and $\fb$, $R/\fa$ is Cohen-Macaulay if and only if $R/\fb$
is. They also gave a counter-example to show that the above
statement is no longer true if the base ring is Cohen-Macaulay but
non-Gorenstein. Attempts to generalize this
theorem lead to several development in linkage theory, especially by
C. Huneke and B. Ulrich (\cite{Hu} and \cite{HuUl}). In \cite[Theorem 4.1]{Sc}, Schenzel proved that,
over a Gorenstein local ring $R$ with maximal ideal $\fm$, the Serre
condition $(S_r)$ for $R/{\fa}$ is equivalent to the vanishing of
the local cohomology groups $\hh^i_{\fm}(R/{\fb})= 0$ for all $i$,
$\dim(R/{\fb})-r<i<\dim(R/{\fb})$, provided $\fa$ an $\fb$ are
linked by a Gorenstein ideal $\fc$.

In \cite{MS}, Martsinkovsky and Strooker generalized the notion of
linkage for modules over non-commutative semiperfect Noetherian
rings. They introduced the operator $\lambda=\Omega\Tr$ and showed
that ideals $\fa$ and $\fb$ are linked by zero ideal if and only if
$R/\fa\cong\lambda (R/\fb)$ and $R/\fb\cong\lambda (R/\fa)$
\cite[Proposition 1]{MS}.

In \cite{DS} and \cite{DS1}, Dibaei and Sadeghi studied the theory of linkage for modules of finite $\gc$-dimension.
In this paper, we continue our study and obtain new results in this direction.

Now we describe the organization of the paper. In Section 2, we collect preliminary notions, definitions and some known results
which will be used in this paper.

In Section 3, for a horizontally linked $R$--module $M$, the connections of its invariants
$\gc$-dimension, depth and relative reduced grade with respect to a semidualizing module are studied. The associated prime ideals of
$\Ext^{\tiny{\rgric(\lambda M)}}_{\Ic}(\lambda M, R)$ is determined in terms of $M$, for any horizontally linked module $M$ (see Lemma \ref{l3}). As a consequence, it is shown that under certain conditions the depth of a horizontally linked module
is equal to the relative reduced grade of its linked module. For a horizontally linked $R$--module $M$ of finite $\gc$-dimension, the connection of the Serre condition $(S_n)$ with the vanishing of certain relative cohomology modules of its linked module is studied.
As a consequence, we obtain the following result for the linkage of ideals (see also Theorem \ref{t6} for more general case).
\begin{thm}\label{th1}
Let $(R,\fm)$ be a local ring, $C$ a semidualizing $R$--module and $\alpha$ a $\gc$-perfect ideal of $R$ of grade $m$. Set $S=R/\alpha$, $K=\Ext^{m}_R(S,C)$. Assume that the ideals $I$ and $J$ are linked by ideal $\alpha$ such that $K/JK$ satisfies $\widetilde{S}_1$. If $\id_{S_\fp}(K_\fp)<\infty$ for all $\fp\in X^{0}(S)$ then the following statements hold true.
\begin{enumerate}[(i)]
\item{$R/I$ is $\gc$-perfect if and only if $\Ext^i_{R/\alpha}(R/I,K)=0=\Ext^i_{\mathcal{I}_K}(R/J,R/\alpha)$ for all $i>0$.}
\item{If $\gkd(R/I)<\infty$, then $R/I$ satisfies $\widetilde{S}_n$ if and only if $\Ext^i_{\mathcal{I}_K}(R/J,R/\alpha)=0$ for all $i$, $0<i<n$.}
\item{If $\depth R/\alpha \geq2$, $(R/I)_\fp$ is $\gcp$-perfect for all $\fp\in\Spec(R)-\{\fm\}$, then $$\Ext^i_{\mathcal{I}_K}(R/J,R/\alpha)\cong\hh^i_{\fm}(R/I)$$ for all $i$, $0<i<\depth R/\alpha$.}
\end{enumerate}
\end{thm}

In Section 4, we study the theory of linkage for modules in the Auslander class. For a semidualizing $R$-module $C$ and a stable $R$--module $M$, it is shown that if $M\in\ac$ and $\id_{R_\fp}(C_{\fp}) <\infty$ for all $\fp\in\Spec R$ with $\depth R_{\fp}\leq n-2$, then $M$ is an $n$-th syzygy module if and only if $M$ is horizontally linked and $\Ext^i_R(\lambda M, C) =0$ for all $i$, $0 <i <n$(see Corollary \ref{cor7}).\\

In Section 5, the connection
of the Serre condition $(S_n)$ on an $R$-module of finite $\gc$-dimension with the vanishing of the local cohomology groups of
its linked module is discussed. Let $R$ be a Cohen-Macaulay local ring $R$ with canonical $\omega_R$ and $C$ a semidualizing $R$--module. For a module $M$ of finite $\gc$-dimension, it is shown that $M$ satisfies the $\widetilde{S}_n$ (i.e. $\depth_{R_{\fp}}(M_\fp)\geq\min\{n,\depth R_\fp\}$ for all $\fp\in\Supp_R(M)$) if and only if $\Ext^i_R(\trk M,\omega_R)=0$ for all $i$, $1\leq i\leq n$ (see Theorem \ref{t5}). As a consequence, we can generalize the Schenzel's result for modules of finite $\gc$-dimension. For a semidualizing module $C$ and a stable module $M$ over a Cohen-Macaulay local ring $R$, it is shown that if $\gkd_R(M)<\infty$ and $\underline{\Hom}_R(M,C)=0$ then $M$ satisfies $\widetilde{S}_n$ if and only if $M$ is horizontally linked and $\hh^i_\fm(\lambda M\otimes_RC)=0$ for all $i$, $\dim R-n<i<\dim R$. As an application of our study, we obtain the following result (see also Corollary \ref{cor9} for more general case).
\begin{thm}\label{th2}
Let $(R,\fm)$ be a Cohen-Macaulay local ring, $C$ a semidualizing $R$--module and $\fc$ a $\gc$-Gorenstein ideal of $R$. Assume that  $M$, $N$ two $R/\fc$--modules such  that $M\underset{\fc}{\thicksim}N$. If $\gkd_R(M)<\infty$, then the following statements hold true.
\begin{enumerate}[(a)]
\item{The following are equivalent:
\begin{enumerate}[(i)]
\item{$M$ is Cohen-Macaulay;}
\item{$N$ is Cohen-Macaulay;}
\item{$M$ satisfies $(S_n)$ for some $n>\depth(R/\fc)-\depth_{R/\fc}(N)$;}
\item{$N$ satisfies $(S_n)$ for some $n>\depth(R/\fc)-\depth_{R/\fc}(M)$.}
\end{enumerate}}
\item{If $\dim R/\fc\geq2$, then $M$ is generalized Cohen-Macaulay if and only if $N$ is. Moreover, if $M$ is generalized Cohen-Macaulay, then
$$\hh^i_\fm(M)\cong\Ext^i_{R/\fc}(N,R/\fc) \text{ for all } i, 0<i<d.$$ In particular, if $M$ is not maximal Cohen-Macaulay then $\depth_R(M)=\rgr_{R/\fc}(N)$.}
\end{enumerate}
\end{thm}

\section{Preliminaries}
Throughout the paper, $R$ is a commutative Noetherian ring and all
$R$--modules $M$, $N$, $\cdots$ are finite (i.e. finitely
generated). Whenever, $R$ is assumed local, its unique maximal ideal
is denoted by $\fm$.

For a finite presentation $P_1\overset{f}{\rightarrow}P_0\rightarrow
M\rightarrow 0$ of an $R$--module $M$, its transpose, $\Tr M$, is
defined as $\coker f^*$, where $(-)^* := \Hom_R(-,R)$, which
satisfies in the exact sequence
\begin{equation}\label{1.1}
0\rightarrow M^*\rightarrow P_0^*\rightarrow P_1^*\rightarrow \Tr
M\rightarrow 0.
\end{equation}
Moreover, $\Tr M$ is unique up to projective equivalence. Thus all
minimal projective presentations of $M$ represent isomorphic
transposes of $M$. The syzygy of a module $M$, denoted by $\Omega M$, is the kernel of
an epimorphism $P\overset{\alpha}{\rightarrow}M$, where $P$ is a
projective $R$--module, so that it is unique up to projective
equivalence. Thus $\Omega M$ is uniquely determined, up to
isomorphism, by a projective cover of $M$.

Martsinkovsky and Strooker  \cite{MS} generalized
the notion of linkage for modules over non-commutative semiperfect
Noetherian rings (i.e. finitely generated modules over such rings have projective covers). They introduced the operator $\lambda:=\Omega\Tr$ and showed that ideals $\fa$ and $\fb$ are linked by zero ideal if and only if $R/\fa\cong\lambda(R/\fb)$ and $R/\fb\cong\lambda(R/\fa)$\cite[Proposition1]{MS}.
\begin{dfn}\label{d1}\cite[Definition 3]{MS}
Let $R$ be a semiperfect ring. Two $R$--modules $M$ and $N$ are said
to be\emph{ horizontally linked} if $M\cong \lambda N$ and
$N\cong\lambda M$. Also, $M$ is called horizontally linked (to
$\lambda M$) if $M\cong\lambda^2M$.
\end{dfn}
Note that a commutative ring $R$ is semiperfect if and only if it is
a finite direct product of commutative local rings \cite[Theorem
23.11]{L}. A \emph{stable} module is a module with no non-zero
projective direct summands. Let $R$ be a semiperfect ring, $M$ a
stable $R$--module and $P_1\rightarrow P_0\rightarrow M\rightarrow
0$ a minimal projective presentation of $M$. Then $P_0^*\rightarrow
P_1^*\rightarrow \Tr M\rightarrow 0$ is a minimal projective
presentation of $\Tr M$ \cite[Theorem 32.13]{AF}. The following
induced exact sequences
\begin{equation}\tag{\ref{d1}.1}
0\longrightarrow M^*\longrightarrow P_0^*\longrightarrow \lambda
M\longrightarrow0,
\end{equation}
\begin{equation}\tag{\ref{d1}.2}
0\longrightarrow \lambda M\longrightarrow P_1^*\longrightarrow \Tr
M\longrightarrow0,
\end{equation}
which will be quoted in this paper.

An $R$--module $M$ is called a \emph{syzygy module} if it is
embedded in a projective $R$--module. Let $i$ be a positive integer,
an $R$--module $M$ is said to be an $i$th syzygy if there exists an
exact sequence
$$0\rightarrow M\rightarrow P_{i-1}\rightarrow\cdots\rightarrow P_0$$ with the $P_0,\cdots,P_{i-1}$ are
projective. By convention, every module is a $0$th syzygy.

Here is a characterization of a horizontally linked module.

\begin{thm}\cite[Theorem 2 and Proposition 3]{MS}\label{MS}
Let $R$ be a semiperfect ring. An $R$--module $M$ is horizontally
linked if and only if it is stable and $\Ext^1_R(\Tr M,R)=0$,
equivalently $M$ is stable and a syzygy module.
\end{thm}
\begin{dfn}\label{d3}
An $R$--module $C$ is called a \emph{semidualizing} module, if the homothety morphism $R\rightarrow\Hom_R(C,C)$ is
an isomorphism and $\Ext^i_R(C,C)=0$ for all $i>0$.
\end{dfn}
Semidualizing modules are initially studied in \cite{F} and
\cite{G}. It is clear that $R$ itself is a semidualizing
$R$--module. Over a Cohen-Macaulay local ring $R$, a canonical
module $\omega_R$ of $R$ is a semidualizing module with finite
injective dimension.

Let $C$ be a semidualizing $R$--module, $M$ an $R$--module. Let
$P_1\overset{f}{\rightarrow}P_0\rightarrow M\rightarrow 0$ be a
projective presentation of $M$. The transpose of $M$ with respect to
$C$, $\trk M$, is defined to be $\coker f^{\triangledown}$, where
$(-)^{\triangledown} := \Hom_R(-,C)$, which satisfies the exact
sequence
\begin{equation}\tag{\ref{d3}.1}
0\rightarrow M^{\triangledown}\rightarrow
P_0^{\triangledown}\rightarrow P_1^{\triangledown}\rightarrow \trk
M\rightarrow 0.
\end{equation}
By \cite[Proposition 3.1]{F}, there exists the following exact
sequence
\begin{equation}\tag{\ref{d3}.2}
0\rightarrow\Ext^1_R(\trk M,C)\rightarrow M\rightarrow
M^{\triangledown\triangledown}\rightarrow\Ext^2_R(\trk
M,C)\rightarrow0.
\end{equation}

The Gorenstein dimension has been extended to $\gc$--dimension by
Foxby in \cite{F} and by Golod in \cite{G}.
\begin{dfn}
An $R$--module $M$ is said to have $\gc$--dimension zero if $M$ is
$C$-reflexive, i.e. the canonical map $M\rightarrow
M^{\triangledown\triangledown}$ is bijective and
$\Ext^i_R(M,C)=0=\Ext^i_R(M^{\triangledown},C)$ for all $i>0$.
\end{dfn}
A $\gc$-resolution of an $R$--module $M$ is a right acyclic complex
of $\gc$-dimension zero modules whose $0$th homology is $M$. The
module $M$ is said to have finite $\gc$-dimension, denoted by
$\gkd_R(M)$, if it has a $\gc$-resolution of finite length.

Note that, over a local ring $R$, a semidualizing $R$--module $C$ is
a canonical module if and only if $\gkd_R(M)<\infty$ for all
finitely generated $R$--modules $M$ (see \cite[Proposition
1.3]{Ge}).

In the following, we summarize some basic facts about $\gc$-dimension (see \cite{G} for more details).
\begin{thm}\label{G3}
For a semidualizing $R$--module $C$ and an $R$--module $M$, the following statements hold true.
\begin{enumerate}[(i)]
       \item{$\gkd_R(M)=0$ if and only if $\Ext^i_R(M,C)=0=\Ext^i_R(\trk M,C)$ for all $i>0$;}
       \item{$\gkd_R(M)=0$ if and only if $\gkd_R(\trk M)=0$;}
       \item{If $\gkd_R(M)<\infty$ then $\gkd_R(M)=\sup\{i\mid\Ext^i_R(M,C)\neq0, i\geq0\}$;}
        \item{If $R$ is local and $\gkd_R(M)<\infty$, then $\gkd_R(M)=\depth R-\depth_R(M)$.}
\end{enumerate}
\end{thm}
\begin{dfn}\label{def1}
Let $C$ be a semidualizing $R$--module. The \emph{Auslander class
with respect to} $C$, denoted $\mathcal{A}_{C}$, consists of all
$R$--modules $M$ satisfying the following conditions.
\begin{itemize}
\item[(i)] The natural map $\mu:M\longrightarrow\Hom_R(C,M\otimes_RC)$ is an isomorphism;
\item[(ii)] $\Tor_i^R(M,C)=0=\Ext^i_R(C,M\otimes_RC)$ for all $i>0$.
\end{itemize}
Dually, The Bass class with respect to $C$, denoted $\bc$, consists of all $R$--modules
$M$ satisfying
\begin{itemize}
\item[(i)] The natural evaluation map $\mu:C\otimes_R\Hom_R(C,M)\longrightarrow M$ is an isomorphism;
\item[(ii)] $\Tor_i^R(\Hom_R(C,M),C)=0=\Ext^i_R(C,M)$ for all $i>0$.
\end{itemize}
\end{dfn}
In the following we collect some basic properties and examples of modules
in the Auslander class and Bass class with respect to a semidualizing module which
will be used in the rest of this paper.
\begin{eg}\label{example1}
\begin{enumerate}[(i)]
  \item{ If any two $R$-modules in a short exact sequence are
     in $\mathcal{A}_C$, respectively $\bc$, then so is the third one \cite[Lemma 1.3]{F}.
     Hence, every module of finite projective dimension
     is in the Auslander class $\mathcal{A}_C$. Also the class $\bc$, contains all modules of finite injective dimension.}
  \item{ Over a Cohen-Macaulay local ring $R$ with canonical module $\omega_R$,
        $M\in\mathcal{A}_{\omega_R}$ if and only if $\gd_R(M)<\infty$
        \cite[Theorem 1]{F1}.}
   \item{The $\mathcal{I}_C$-injective dimension of $M$, denoted
    $\mathcal{I}_{C}$-$\id_R(M)$, is less than or equal to $n$ if and
           only if there is an exact sequence
            $$0\rightarrow M\rightarrow\Hom_R(C,I^0)\rightarrow\cdots\rightarrow\Hom_R(C,I^n)\rightarrow0,$$
             such that each $I^i$ is an injective $R$--module \cite[Corollary 2.10]{TW}.
              Note that if $M$ has a finite $\mathcal{I}_C$-injective
              dimension, then $M\in\mathcal{A}_C$
               \cite[Corollary 2.9]{TW}.}
    \item{Let $R$ be a Cohen-Macaulay local ring with canonical module $\omega_R$ and let $C$ be a semidualizing $R$--module. It is well-known that $C^{\dagger}:=\Hom_R(C,\omega_R)$ is a semidualizing $R$--module. For an $R$--module $M$, $\gkd_R(M)<\infty$ if and only if $M\in\mathcal{A}_{C^{\dagger}}$\cite[Theorem 4.6]{HJ}.}
\end{enumerate}
\end{eg}
The class of $C$--injective modules is defined as follows
$$\Ic=\{\Hom_R(C,I)\mid I \text{ is injective}\}.$$
The class $\Ic$ is preenveloping. Then each $R$--module has an augmented proper
$\Ic$-coresolution, that is, an $R$-complex:
$$Y^{+}=0\rightarrow M\rightarrow Y_0\rightarrow Y_1\rightarrow\cdots$$
such that $\Hom_R(Y^{+},X)$ is exact for all $X\in\Ic$. The truncated complex
$$0\rightarrow Y_0\rightarrow Y_1\rightarrow\cdots$$
is a proper $\Ic$-coresolution of $M$.

\begin{dfn}\cite{TW}
Let $C$ be a semidualizing $R$--module, and let $M$ and $N$ be
$R$--modules. Let $J$ be a proper $\Ic$-coresolution of $N$. For each $i$, set
$$\Ext^i_{\Ic}(M,N):=\hh_{-i}(\Hom_R(M,J)).$$
\end{dfn}
\begin{thm}\label{t2}\cite[Theorem 4.1 and Corollary 4.2]{TW}
Let $C$ be a semidualizing $R$--module, and let $M$ and $N$ be $R$--modules.
There exists isomorphism
$$\Ext^i_{\Ic}(M,N)\cong\Ext^i_R(C\otimes_RM,C\otimes_RN)$$ for all $i\geq0$. Moreover, if $M$ and $N$ are in $\mathcal{A}_C$ then
$\Ext^i_{\Ic}(M,N)\cong\Ext^i_R(M,N)$ for all $i\geq0$.
\end{thm}
The \emph{relative reduced grade} of a module $M$ with respect to a semidualizing $C$ defined as follows
$$\rgric(M):=\inf\{i>0\mid\Ext^i_{\Ic}(M,R)\neq0\}.$$
We denote by $\rgr(M)$ the reduced grade of $M$ with respect to $R$. Note that if $M\in\ac$ then $\rgric(M)=\rgr(M)$ by Theorem \ref{t2}.

\begin{dfn}\label{def2}
Let $M$ and $N$ be finitely generated $R$--modules. Denote by
$\mathcal{\beta}(M, N)$ the set of $R$--homomorphisms of $M$ to $N$ which pass through free modules.
That is, an $R$-homomorphism $f:M\rightarrow N$ lies in $\mathcal{\beta}(M, N)$ if and only if it is factored as
$M\rightarrow F\rightarrow N$ with $F$ free. We denote
$$\underline{\Hom}_R(M,N)=\Hom_R(M,N)/\mathcal{\beta}(M,N).$$
By \cite[Lemma 3.9]{y}, there is a natural isomorphism
\begin{equation}\tag{\ref{def2}.1}
\underline{\Hom}_R(M,N)\cong\Tor_1^R(\Tr M,N).
\end{equation}
\end{dfn}

Throughout, $C$ is a semidualizing $R$--module and we denote $(-)^\triangledown$
as the dual functor $(-)^\triangledown:=\Hom_R(-,C)$. Two
$R$--modules $M$ and $N$ are said to be stably equivalent with
respect $C$, denoted $M \underset{C}{\approx} N$, if $C^n\oplus
M\cong C^m\oplus N$ for some non-negative integers $m$ and $n$. We
write $M \approx N$ when $M$ and $N$ are stably equivalent with
respect $R$. For $k>0$, set $\mathcal{T}_k^{C}:=\trk\Omega^{k-1}$.

\begin{rmk}\label{remark3}
\begin{enumerate}[(i)]
\item{Let $R$ be a semiperfect ring, $M$ an $R$--module. Assume that
$P_1\overset{f}{\rightarrow} P_0\rightarrow M\rightarrow0$ is the minimal
projective presentation of $M$. There exists a commutative diagram
$$\begin{CD}
&&&&&&&&\\
  \ \ &&&&  P^*_0\underset{R}{\otimes}C @>>>P^*_1\underset{R}{\otimes}C @>>> \Tr M\underset{R}{\otimes}C @>>>0&  \\
  &&&& @VV{\cong}V @VV{\cong}V \\
  \ \  &&&& P^{\triangledown}_0 @>f^{\triangledown}>> P^{\triangledown}_1 @>>>\trk M @>>>0&\\
\end{CD}$$\\
with exact rows. Therefore, $\Tr M\underset{R}{\otimes}C\cong\trk
M$.}
\item{Set $\lambda_{C}M:=\im f^{\triangledown}$. If $\underline{\Hom}_R(M,C)=0$, then $\Tor_1^R(\Tr M,C)=0$ by (\ref{def2}.1). Now it is easy to see that $\lambda_{C}M\cong\lambda M\otimes_R C$ and so we have the following isomorphisms:
\[\begin{array}{rl}
\Ext^i_{\Ic}(\lambda M,R)&\cong\Ext^i_R(\lambda M\otimes_RC,C)\\
&\cong\Ext^i_R(\lambda_C M,C)\\
&\cong\Ext^{i+1}_R(\trk M,C),\\
\end{array}\]
for all $i>0$, where the first isomorphism follows from Theorem \ref{t2}.}
\item{ For an $R$--module $M$ and positive integer $n$, there
are exact sequences
$$0\rightarrow\Ext^n_R(M,C)\rightarrow\mathcal{T}_n^{C}M\rightarrow L\rightarrow0,$$
$$0\rightarrow L\rightarrow \overset{m}{\oplus} C\rightarrow \mathcal{T}_{n+1}^CM\rightarrow0$$
}
\end{enumerate}
\end{rmk}
In \cite{HH}, C. Huang and Z. Huang introduced Gorenstein transpose of a module and investigated the
relations between the Gorenstein transpose and the transpose of the same module. In \cite{LY}, Liu and Yan extended
the notion of Gorenstein transpose to $C$-Gorenstein transpose as follows.\\

For a finite $\gc$-projective presentation $\pi: X_1\overset{f}{\rightarrow}X_0\rightarrow
M\rightarrow 0$ of an $R$--module $M$, its $C$-Gorenstein transpose, $\Tr^{\pi}_{\gc} M$, is
defined as $\coker f^{\triangledown}$, where $(-)^{\triangledown} := \Hom_R(-,C)$, which
satisfies in the exact sequence
\begin{equation}
0\rightarrow M^\triangledown\rightarrow X_0^\triangledown\rightarrow X_1^\triangledown\rightarrow \Tr^{\pi}_{\gc} M\rightarrow 0.
\end{equation}

\begin{lem}\label{l4}
For an $R$--module $M$, there exists the following exact sequence
$$0\rightarrow M\rightarrow\trk(\trk M)\rightarrow X\rightarrow0,$$
where $\gkd_R(X)=0$.
\end{lem}
\begin{proof}
Let $P_1\rightarrow P_0\rightarrow M\rightarrow0$ be a projective presentation of $M$. By applying the functor
$(-)^{\triangledown}= \Hom_R(-,C)$, on the projective presentation of $M$, we obtain a $\gc$-projective presentation of $\trk M$,
$\pi: (P_0)^{\triangledown}\rightarrow (P_1)^{\triangledown}\rightarrow \trk M\rightarrow0.$
By applying the functor $(-)^{\triangledown}$, on $\pi$ we obtain the following commutative diagram:
$$\begin{CD}
&&&&&&&&\\
  \ \ &&&&  P_1 @>>>P_0 @>>>  M @>>>0&  \\
  &&&& @VV{\cong}V @VV{\cong}V \\
  \ \  &&&& P^{\triangledown\triangledown}_1 @>>> P^{\triangledown\triangledown}_0 @>>>\Tr_{\gc}^{\pi}(\trk M) @>>>0&\\
\end{CD}$$\\
with exact rows. Therefore, $M\cong\Tr_{\gc}^{\pi}(\trk M)$. Now the assertion is clear by \cite[Proposition 4.2]{LY}.
\end{proof}
Let $M$ be an $R$--module. For a
generating set $\{f_1,f_2,\ldots,f_n\}$ of
$M^\triangledown=\Hom_R(M,C)$, denote $f:M\rightarrow C^n$ as the
map $(f_1,\ldots,f_n)$. It follows from (\ref{d3}.2) that $f$ is
injective if and only if $\Ext^1_R(\trk M,C)=0$. Note that when $f$
is injection, then there is an exact sequence
$$0\rightarrow M\overset{f}{\rightarrow} C^n\rightarrow N\rightarrow0,$$
where $N=\coker(f)$. It is easy to see that, in this situations, the
above exact sequence is dual exact with respect to
$(-)^\triangledown$ and so $\Ext^1_R(N,C)=0$. Such an exact sequence
is called a \emph{universal pushforward of $M$ with respect to $C$}.

\begin{dfn}\cite{M1}\label{S}
An $R$--module $M$ is said to satisfy the property $\widetilde{S}_k$
if $\depth_{R_\fp} (M_\fp) \geq  \min\{ k, \depth R_\fp\}$  for all
$\fp\in\Spec R$.
\end{dfn}
Note that, for a horizontally linked module $M$ over a
Cohen-Macaulay local ring $R$, the properties $\widetilde{S}_k$ and
$(S_k)$ are identical.

For a positive integer $n$, a module $M$ is called an $n$th
$C$-syzygy module if there is an exact sequence $0\rightarrow
M\rightarrow C_1\rightarrow C_2\rightarrow\ldots\rightarrow C_n$,
where $C_i\cong\oplus^{m_i}C$ for some $m_i$.

Let $X$ be a subset of $\Spec R$. An $R$--module $M$ is said to be
of finite $\gc$--dimension on $X$, if $\gcpd(M_{\fp})<\infty$ for
all $\fp\in X$. We denote $X^n(R):=\{\fp\in\Spec(R)\mid \depth
R_\fp\leq n\}$.

Recall that an $R$--module $M$ is $n$-torsion free if $\Ext^i_R(\Tr
M,R)=0$ for all $i$, $1\leq i\leq n$. In \cite[Theorem 4.25]{AB},
Auslander and Bridger proved that an $R$--module $M$ of finite
Gorenstein dimension is $n$-torsion free if and only if $M$
satisfies $\widetilde{S}_n$ (See also \cite[Theorem 42]{M1}). In \cite[Proposition 2.4]{DS1}, Dibaei and Sadeghi generalized this result as follows.
\begin{thm}\label{t1}
Let $M$ be an $R$--module. For a
positive integer $n$, consider the following statements.
\begin{itemize}
      \item[(i)]$\Ext^i_R(\trk M,C)=0$ for all $i$, $1\leq i\leq n$.
      \item[(ii)]$M$ is an $n$th $C$-syszygy module.
      \item[(iii)]$M$ satisfies $\widetilde{S}_n$.
\end{itemize}
Then the following implications hold true.
\begin{itemize}
       \item[(a)] (i)$\Rightarrow$(ii)$\Rightarrow$(iii).
       \item[(b)] If $M$ has finite $\gc$--dimension on $X^{n-1}(R)$, then (iii)$\Rightarrow$(i).
\end{itemize}
\end{thm}
\begin{rmk}\label{rem2}
\begin{enumerate}[(i)]
\item{By \cite[Theorem 10.62]{R}, there is a third quadrant spectral sequence
$$\E^{p,q}_2=\Ext^p_R(\Tor_q^R(\Tr M,C),C)\Rightarrow\Ext^{p+q}_R(\Tr M,R).$$
Hence we obtain the following exact sequence $$0\rightarrow\Ext^1_R(\Tr M\otimes_RC,C)\rightarrow\Ext^1_R(\Tr M,R),$$
by \cite[Theorem 10.33]{R}. If $M$ is horizontally linked, then $\Ext^1_R(\Tr M,R)=0$. Therefore $\Ext^1_R(\trk M,C)\cong\Ext^1_R(\Tr M\otimes_RC,C)=0$.}
\item{If  $\Tr M\in\mathcal{A}_C$ then
$\Tor_i^R(\Tr M,C)=0$ for all $i>0$. Hence $\E^{p,q}_2=0$ for all
$q>0$. Therefore, the spectral sequence collapses on $p$-axis and so
$$\Ext^i_R(\trk M,C)\cong\Ext^i_R(\Tr M\otimes_RC,C)\cong\Ext^i_R(\Tr
M,R),$$ for all $i\geq0$.}
\end{enumerate}
\end{rmk}

\begin{lem}\cite[Lemma 2.11]{DS1}\label{lem2}
Let $R$ be a local ring, $n\geq 0$
an integer, and $M$ an $R$--module. If $M\in\mathcal{A}_C$, then the
following statements hold true.
\begin{enumerate}[(i)]
\item $\depth_R(M)=\depth_R(M\otimes_RC)$ and
$\dim_R(M)=\dim_R(M\otimes_RC)$;
\item $M$ satisfies $(S_n)$ if and only if $M\otimes_RC$ does;
\item $M$ is Cohen-Macaulay if and only if $M\otimes_RC$ is Cohen-Macaulay.
\end{enumerate}
\end{lem}
\section{Linkage of modules and relative cohomology}
In this section, for a horizontally linked $R$--module $M$, the connections of its invariants
$\gc$-dimension, depth and relative reduced grade with respect to a semidualizing module are studied.
In the following, we investigate the relation between the Serre condition $\widetilde{S}_n$ on a horizontally linked module with the vanishing of certain relative cohomology modules of its linked module.
\begin{prop}\label{t3}
Let $R$ be a semiperfect ring, $M$ a horizontally linked $R$--module and $\underline{\Hom}_R(M,C)=0$.
The following statements hold true:
\begin{enumerate}[(i)]
\item{$\gkd_R(M)=0$ if and only if $\Ext^i_R(M,C)=0=\Ext^i_{\Ic}(\lambda M,R)$ for all $i>0$;}
\item{If $\gcpd(M_\fp)<\infty$ for all $\fp\in X^{n-1}(R)$,
then $M$ satisfies $\widetilde{S}_n$ if and only if $\Ext^i_{\Ic}(\lambda M,R)=0$ for all $i$,
$0<i<n$.}
\end{enumerate}
\end{prop}
\begin{proof}
As $M$ is horizontally linked, $\Ext^1_R(\trk M,C)=0$ by Remark \ref{rem2}(i). We have the following isomorphism:
\begin{equation}\tag{\ref{t3}.1}
\Ext^i_{\Ic}(\lambda M,R)\cong\Ext^{i+1}_R(\trk M,C),
\end{equation}
for all $i>0$ by Remark \ref{remark3}(ii).

(i) By Theorem \ref{G3}(i), $\gkd_R(M)=0$ if and only if $\Ext^i_R(M,C)=0=\Ext^i_R(\trk M,C)$ for all $i>0$ and this is equivalent to say that
$\Ext^i_R(M,C)=0=\Ext^i_{\Ic}(\lambda M,R)$ for all $i>0$ by (\ref{t3}.1).

(ii) By Theorem \ref{t1}, $M$ satisfies $\widetilde{S}_n$
if and only if $\Ext^i_R(\trk M,C)=0$ for all $i$, $1\leq i\leq n$ and this is equivalent to say that
$\Ext^i_{\Ic}(\lambda M,R)=0$ for all $i$, $0<i<n$ by (\ref{t3}.1).
\end{proof}
\begin{cor}\label{cor10}
Let $R$ be a local ring and $M$ a horizontally linked $R$--module of finite $\gc$-dimension. If $\underline{\Hom}_R(M,C)=0$, then $\gkd_R(M)=0$ if and only if $\Ext^i_{\Ic}(\lambda M,R)=0$ for all $i$, $1\leq i\leq\depth(R)$.
\end{cor}
\begin{proof}
This is an immediate consequence of Proposition \ref{t3} and Theorem \ref{G3}(iv).
\end{proof}
\begin{cor}
Let $R$ be a Cohen-Macaulay local ring with canonical module $\omega_R$ and $M$ a horizontally linked $R$--module. If $\underline{\Hom}_R(M,\omega_R)=0$, then $M$ is maximal Cohen-Macaulay if and only if
$\Ext^i_{{\mathcal{I}}_{\omega_R}}(\lambda M,R)=0$ for all $i$, $1\leq i\leq\dim (R)$.
\end{cor}
In the following, for a horizontally linked module
$M$ of finite and positive $\gc$-dimension, we express the associated primes of the
$\Ext^{\tiny{\rgric(\lambda M)}}_{\Ic}(\lambda M,R)$ in terms of $M$.
\begin{lem}\label{l3}
Let $R$ be a semiperfect ring and let $M$ be a horizontally linked module of finite and positive $\gc$-dimension. Set $n=\rgric(\lambda M)$. If $\underline{\Hom}_R(M,C)=0$ then
$$\Ass_R(\Ext^n_{\Ic}(\lambda M,R))=\{\fp\in\Spec(R)\mid \gcpd(M_\fp)\neq0, \depth_{R_\fp}(M_\fp)=n=\rgr_{\mathcal{I}_{C_\fp}}((\lambda M)_\fp)\}.$$
\end{lem}
\begin{proof}
As $M$ is horizontally linked, $\Ext^1_R(\trk M,C)=0$ by Remark \ref{rem2}(i). We have the following isomorphism:
\begin{equation}\tag{\ref{l3}.1}
\Ext^i_{\Ic}(\lambda M,R)\cong\Ext^{i+1}_R(\trk M,C),
\end{equation}
for all $i>0$ by Remark \ref{remark3}(ii).
Set $N=\trk M$. Let $\cdots\rightarrow P_i\rightarrow\cdots\rightarrow
P_0\rightarrow N\rightarrow0$ be the minimal projective resolution
of $N$. As $n=\rgric(\lambda M)$, $\Ext^i_R(N,C)=0$ for all $i$, $1\leq i\leq n$ by (\ref{l3}.1). Hence by
applying functor
$(-)^{\triangledown}=\Hom_R(-,C)$ on the minimal projective
resolution of $N$, we obtain the following exact sequences:
\begin{equation}\tag{\ref{l3}.2}
0\rightarrow\trk N\rightarrow
(P_2)^{\triangledown}\rightarrow\cdots\rightarrow(P_{n+1})^{\triangledown}\rightarrow\trc_{n+1}N\rightarrow0,
\end{equation}
\begin{equation}\tag{\ref{l3}.3}
0\rightarrow\Ext^{n+1}_R(N,C)\rightarrow\trc_{n+1}N\rightarrow L\rightarrow0,
\end{equation}
\begin{equation}\tag{\ref{l3}.4}
0\rightarrow L\rightarrow
\overset{m}{\oplus}C\rightarrow\trc_{n+2}N\rightarrow0.
\end{equation}
Assume that $\fp\in\Ass_R(\Ext^{n}_{\Ic}(\lambda M,C))$. It is clear that $n=\rgr_{\mathcal{I}_{C_\fp}}((\lambda M)_\fp)$. By (\ref{l3}.1), $\fp\in\Ass_R(\Ext^{n+1}_{R}(N,C))$. It follows from the exact sequence (\ref{l3}.3) that $\depth_{R_\fp}(\trc_{n+1}N)_\fp=0$. If $\gcpd(M_\fp)=0$, then $\Ext^i_{R_\fp}(N_\fp,C_\fp)=0$ for all $i>0$ by Theorem \ref{G3}(i) which is a contradiction.
Hence $\gcpd(M_\fp)\neq0$. By Proposition \ref{t3}, $M$ satisfies $\widetilde{S}_n$ and so $\depth R_\fp >n$.
Note that $\depth_{R_\fp}(((P_i)^{\triangledown})_\fp)=\depth_{R_\fp}(C_\fp)=\depth
R_\fp$ for all $i$. By localizing the exact sequence (\ref{l3}.2)
at $\fp$, we conclude that $\depth_{R_\fp}((\trk N)_\fp)=n$. By Lemma \ref{l4}, we have the following exact sequence
\begin{equation}\tag{\ref{l3}.5}
0\rightarrow M\rightarrow\trk N\rightarrow X\rightarrow0,
\end{equation}
where $\gkd_R(X)=0$. By localizing the exact sequence (\ref{l3}.5)
at $\fp$, we conclude that $\depth_{R_\fp}(M_\fp)=n$.

Now assume that $\fp\in\Spec R$ such that $\gcpd(M_\fp)\neq0$ and
$\depth_R( M_\fp)=n=\rgr_{\mathcal{I}_{C_\fp}}((\lambda M)_\fp)$. It follows from Theorem \ref{G3}(iv) that $\depth R_\fp>n$. By localizing the exact sequence (\ref{l3}.5) at $\fp$, we conclude that $\depth_{R_\fp}((\trk N)_\fp)=n$. By localizing the exact sequence (\ref{l3}.2) at $\fp$, it is easy to see that $\depth_{R_\fp}((\trc_{n+1}N)_\fp)=0$. As
$\depth_{R_\fp}(C_\fp)=\depth R_\fp>0$, we conclude from the exact
sequence (\ref{l3}.4) that $\depth_{R_\fp}(L_\fp)>0$. It follows
from the exact sequence (\ref{l3}.3) that
$\depth_{R_\fp}(\Ext^{n+1}_R(N,C)_\fp)=0$. In other words,
$\fp\in\Ass_R(\Ext^{n+1}_R(N,C))=\Ass_R(\Ext^{n}_{\Ic}(\lambda M,C))$.
\end{proof}
\begin{prop}\label{p1}
Let $R$ be a semiperfect ring and let $M$ be a horizontally linked $R$--module of finite $\gc$-dimension. If $\underline{\Hom}_R(M,C)=0$, then
$$\rgric(\lambda M)=\inf\{\depth_{R_\fp}(M_\fp)\mid \fp\in\Spec(R), \gcpd(M_\fp)\neq0\}.$$
\end{prop}
\begin{proof}
By Proposition \ref{t3}(i), we may assume that $\gkd_R(M)>0$. Set $n=\rgric(\lambda M)$. By Proposition
\ref{t3}(ii), $M$ satisfies $\widetilde{S}_{n}$. Hence,
\begin{equation}\tag{\ref{p1}.1}
\depth_{R_\fp}( M_\fp)\geq\min\{\depth R_\fp,n\}
\end{equation}
for all $\fp\in\Spec R$. Now let $\fp\in\Spec R$ with $\gcpd(M_\fp)\neq0$. It follows from (\ref{p1}.1) and Theorem \ref{G3}(iv) that $\depth R_\fp>n$.
Therefore, we have
$n\leq\inf\{\depth_{R_\fp}(M_\fp)\mid\fp\in\Spec R,
 \gd_{R_\fp}(M_\fp)\neq0\}.$

On the other hand, by the Lemma \ref{l3}, if
$\fp\in\Ass_R(\Ext^n_{\Ic}(\lambda M,R))$ then $\gcpd(M_\fp)\neq0$ and
 $\depth_{R_\fp}(M_\fp)=n$ and so the assertion holds.
\end{proof}
The reduced grade of a module $M$ with respect to a semidualizing $C$ defined as follows
$$\rgr(M,C):=\inf\{i>0\mid\Ext^i_{R}(M,C)\neq0\}.$$
If $\Ext^i_R(M,C)=0$ for all $i>0$, then $\rgr(M,C)=+\infty$.
Note that if $M$ has a finite and positive $\gc$-dimension, then
$\rgr(M,C)\leq\gkd_R(M)$ by Theorem \ref{G3}(iii).
For a subset $X$ of $\Spec R$, we say that $M$ is of $\gc$-dimension
zero on $X$, if $\gcpd(M_\fp)=0$ for all $\fp$ in $X$.
The following result is a generalization of \cite[Proposition 4.14]{DS1}
\begin{prop}\label{p2}
Let $R$ be a semiperfect ring and let $M$ be a horizontally linked $R$--module of finite and positive $\gc$-dimension such that $\underline{\Hom}_R(M,C)=0$. Set $n=\rgr(M,C)+\rgric(\lambda M)$. Then $M$ is of $\gc$-dimension zero on $X^{n-1}(R)$.
\end{prop}
\begin{proof}
Set $m=\rgric(\lambda M)$. Assume contrarily that $\gcpd(M_\fp)\neq0$
for some $\fp\in X^{n-1}(R)$. By Proposition \ref{t3}(ii), $M$ satisfies $\widetilde{S}_m$. Hence,
\begin{equation}\tag{\ref{p2}.1}
\depth_{R_\fp}(M_\fp)\geq \min\{\depth R_\fp,m\}
\end{equation}
for all $\fp\in\Spec R$. If $\depth R_\fp\leq m$, then
$\gcpd(M_\fp)=\depth R_\fp-\depth_{R_\fp}(M_\fp)=0$ by (\ref{p2}.1), which is a
contradiction. If $\depth R_\fp>m$, then $\depth_{R_\fp}(M_\fp)\geq
m$ by (\ref{p2}.1). Therefore,
$$n-m\leq\rgr(M_\fp,C_\fp)\leq\gcpd(M_\fp)=\depth R_\fp-\depth_{R_\fp}(M_\fp)\leq n-m-1,$$
which is a contradiction.
\end{proof}
For an $R$--module $M$, set
$$\ng(M)=\{\fp\in\Spec(R)\mid\gcpd(M_{\fp})\neq0\}.$$  The following result is a
generalization of \cite[Theorem 2.7]{DS}.
\begin{prop}\label{th3}
Let $R$ be a local ring, $M$ a horizontally linked $R$--module with
$0<\gkd_R(M)<\infty$. If $\underline{\Hom}_R(M,C)=0$ then the following conditions are equivalent.
\begin{enumerate}[(i)]
          \item {$\depth_R(M)=\rgric(\lambda M)$;}
          \item {$\fm\in\Ass_R(\Ext^{\tiny{\rgric(\lambda M)}}_{\Ic}(\lambda M,R))$;}
          \item {$\depth_R(M)\leq\depth_{R_{\fp}}(M_{\fp})$, for each $\fp\in\ng(M)$}.
\end{enumerate}
\end{prop}
\begin{proof}
Set $t=\rgr_{\Ic}(\lambda M)$. By Proposition \ref{t3}(ii), $M$ satisfies $\widetilde{S}_t$. Hence,
\begin{equation}\tag{\ref{th3}.1}
\depth_{R_\fp}(M_\fp)\geq\min\{t, \depth R_\fp\}
\end{equation}
for all $\fp\in\Spec R$. Therefore  $t<\depth R_\fp$ for all $\fp\in\ng(M)$.

$(i)\Rightarrow (ii)$. This follows from Lemma \ref{l3}.

$(ii)\Rightarrow (iii)$. Let $\fp\in\ng(M)$. By Lemma \ref{l3} and (\ref{th3}.1),
$$\depth_R(M)=\rgr_{\Ic}(\lambda M)\leq\depth_{R_\fp}(M_\fp).$$

$(iii)\Rightarrow (i)$. By (\ref{th3}.1), $\depth_R(M)\geq t$. On the other hand, if $\fp\in\Ass_R(\Ext^{t}_{\Ic}(\lambda M,R))$ then $\fp\in\ng(M)$, $\depth_{R_\fp}(M_\fp)=t$ by Lemma \ref{l3}. Hence $\depth_R(M)\leq t$.
\end{proof}
\begin{thm}\label{t4}
Let $(R,\fm)$ be a local ring of depth $d\geq2$ and let $M$ be a horizontally linked $R$--module such that $\underline{\Hom}_R(M,C)=0$. The following statements hold true.
\begin{enumerate}[(i)]
\item{If $\gcpd(M_\fp)=0$ for all $\fp\in\Spec(R)-\{\fm\}$ then $\ell(\Ext^i_{\Ic}(\lambda M,R))<\infty$ for all $i>0$. The converse is true if $\gkd_R(M)<\infty$;}
\item{If $\gcpd(M_\fp)=0$ for all $\fp\in\Spec(R)-\{\fm\}$, then
$$\Ext^i_{\Ic}(\lambda M,R)\cong\hh^i_\fm(M)$$
for all $i$, $0<i<d$. In particular, if $0<\gkd_R(M)<\infty$ then $\rgric(\lambda M)=\depth_R(M)$.}
\end{enumerate}
\end{thm}
\begin{proof}
As $M$ is horizontally linked, $\Ext^1_R(\trk M,C)=0$ by Remark \ref{rem2}(i). We have the following isomorphism:
\begin{equation}\tag{\ref{t4}.1}
\Ext^i_{\Ic}(\lambda M,R)\cong\Ext^{i+1}_R(\trk M,C),
\end{equation}
for all $i>0$ by Remark \ref{remark3}(ii).

(i). If $\gcpd(M_\fp)=0$ for all $\fp\in\Spec(R)-\{\fm\}$, then $\Ext^i_R(\trk M,C)_\fp=0$ for all $i>0$ and all $\fp\in\Spec(R)-\{\fm\}$ by Theorem \ref{G3}(i). It follows from (\ref{t4}.1) that $\ell_R(\Ext^i_{\Ic}(\lambda M,R))<\infty$ for all $i>0$. On the other hand, if $\ell_R(\Ext^i_{\Ic}(\lambda M,R))<\infty$ for all $i>0$ then $\Ext^i_R(\trk M,C)_\fp=0$ for all $i>0$ and all $\fp\in\Spec(R)-\{\fm\}$ by using (\ref{t4}.1) again. By Theorem \ref{t1}, $M_\fp$ satisfies $\widetilde{S}_n$ for all $n$ and all $\fp\in\Spec(R)-\{\fm\}$. Hence, if $\gkd_R(M)<\infty$ then $\gcpd(M_\fp)=\depth R_\fp-\depth_{R_\fp}(M_\fp)=0$ for all $\fp\in\Spec(R)-\{\fm\}$ by Theorem \ref{G3}(iv).

(ii). Set $N=\trk M$. By Remark \ref{remark3}(iii), we have the following exact sequences:
\begin{equation}\tag{\ref{t4}.2}
0\rightarrow\Ext^i_R(N,C)\rightarrow\trc_i N\rightarrow L_i\rightarrow0,
\end{equation}
\begin{equation}\tag{\ref{t4}.3}
0\rightarrow L_i\rightarrow\overset{n_i}{\oplus} C\rightarrow\trc_{i+1}N\rightarrow0,
\end{equation}
for all $i>0$. By part (i) and (\ref{t4}.1), $\Ext^i_R(N,C)$ is of finite length for all $i>0$. By
applying the functor $\Gamma_{\fm}(-)$ on the exact sequences (\ref{t4}.2) and (\ref{t4}.3), we get
\begin{equation}\tag{\ref{t4}.4}
\hh^j_{\fm}(\trc_{i-1}N)\cong\hh^j_{\fm}(L_{i-1}) \text{ for all}\ i\ \ \text{and}\ j, \text{with}\ 1\leq j<d,\ i\geq2,
\end{equation}
\begin{equation}\tag{\ref{t4}.5}
\Ext^i_R(N,C)=\Gamma_{\fm}(\Ext^i_R(N,C))\cong\Gamma_{\fm}
(\trc_iN) \text{ for all}\ i\geq2,
\end{equation}
and also
\begin{equation}\tag{\ref{t4}.6}
\hh^j_{\fm}(\trc_iN)\cong\hh^{j+1}_{\fm}(L_{i-1})
\text{ for all}\ i\ \text{and}\ j,  0\leq j<d-1,\  i\geq2.
\end{equation}
By Lemma \ref{l4}, we have the following exact sequence
$$0\rightarrow M\rightarrow\trc_1N\rightarrow X\rightarrow0,$$
where $\gkd_R(X)=0$. By applying the functor $\Gamma_{\fm}(-)$ on the above exact sequence, we get the following isomorphism
\begin{equation}\tag{\ref{t4}.7}
\hh^j_{\fm}(\trc_1N)\cong\hh^{j}_{\fm}(M)
\text{ for all } j,  0\leq j<d.
\end{equation}
Now by using (\ref{t4}.1), (\ref{t4}.4), (\ref{t4}.5), (\ref{t4}.6) and (\ref{t4}.7) we obtain the result.
\end{proof}
As an immediate consequence, we have the following results.
\begin{cor}
Let $(R,\fm)$ be a Cohen-Macaulay local ring of dimension $d\geq2$ with canonical module $\omega_R$ and let $M$ be a horizontally linked $R$--module such that $\underline{\Hom}_R(M,\omega_R)=0$. The following statements hold true.
\begin{enumerate}
\item{$M$ is generalized Cohen-Macaulay if and only if $\ell(\Ext^i_{\mathcal{I}_\omega}(\lambda M,R))<\infty$ for all $i>0$;}
\item{If $M$ is generalized Cohen-Macaulay, then $$\Ext^i_{\mathcal{I}_\omega}(\lambda M,R)\cong\hh^i_\fm(M)$$ for all $i$, $0<i<d$.}
\end{enumerate}
\end{cor}
\begin{proof}
By \cite[Lemma 1.2 , Lemma 1.4]{T}, $M$ is generalized Cohen-Macaulay module if and only if $M_\fp$ is a maximal Cohen-Macaulay module $R_\fp$--module for all $\fp\in\Spec(R)-\{\fm\}$. Now the assertion is clear by Theorem \ref{t4}.
\end{proof}
\begin{cor}\label{cor1}
Let $(R,\fm)$ be a local ring of depth $d\geq2$ and let $M$ be a stable $R$--module. Assume that $\gcpd(M_\fp)=0$ for all $\fp\in\Spec(R)-\{\fm\}$ and that $\lambda M\in\mathcal{A}_C$, then the following statements hold true.
\begin{enumerate}[(i)]
\item{$M$ is horizontally linked if and only if $\Gamma_{\fm}(M)=0$.}
\item{If $M$ is horizontally linked, then $\Ext^i_R(\lambda M,R)\cong\hh^i_\fm(M)$
for all $i$, $0<i<d$. In particular, if $0<\gkd_R(M)<\infty$ then $\rgr(\lambda M)=\depth_R(M)$.}
\end{enumerate}
\end{cor}
\begin{proof}
(i) As $M$ is locally $\gc$-dimension zero on the punctured spectrum, it follows from Theorem \ref{G3}(i) that $\Supp_R(\Ext^1_R(\trk
M,C))\subseteqq\{\fm\}$. From the exact sequence (\ref{d3}.2), we obtain the following exact sequences:
\begin{equation}\tag{\ref{cor1}.1}
0\rightarrow\Ext^1_R(\trk M,C)\rightarrow M\rightarrow L\rightarrow 0,
\end{equation}
\begin{equation}\tag{\ref{cor1}.2}
0\rightarrow L\rightarrow \oplus C
\end{equation}
As $\depth_R(C)=d\geq2$, it follows from the exact sequence (\ref{cor1}.2) that $\Gamma_{\fm}(L)=0$.
By applying the functor $\Gamma_{\fm}(-)$ on the exact sequence (\ref{cor1}.1), we get
\begin{equation}\tag{\ref{cor1}.3}
\Ext^1_R(\trk M,C)\cong\Gamma_{\fm}(\Ext^1_R(\trk M,C))\cong\Gamma_{\fm}(M)
\end{equation}
Now the assertion is clear by (\ref{cor1}.3), Theorem \ref{MS} and Remark \ref{rem2}(ii).

(ii). By Example \ref{example1}(i), $\Tr M\in\ac$ and so $\underline{\Hom}_R(M,C)=0$ by (\ref{def2}.1).
Now the assertion is clear by Theorem \ref{t4} and Theorem \ref{t2}.
\end{proof}

The following lemma will be useful for the rest of the paper.
\begin{lem}\label{l6}
Let $R$ be a semiperfect ring and $M$ a horizontally linked $R$--module. If $\id_{R_\fp}(C_\fp)<\infty$ for $\fp\in X^0$, then the following are equivalent:
\begin{enumerate}[(i)]
\item{$\underline{\Hom}_R(M,C)=0$;}
\item{$\lambda M\otimes_RC$ satisfies $\widetilde{S}_1$.}
\end{enumerate}
\end{lem}
\begin{proof}
Consider the exact sequence $0\rightarrow\lambda M\rightarrow F\rightarrow\Tr M\rightarrow0$. By applying the functor $-\otimes_R C$ on the above exact sequence
we obtain the following exact sequence:
\begin{equation}\tag{\ref{l6}.1}
0\rightarrow\Tor_1^R(\Tr M,C)\rightarrow\lambda M\otimes_RC\rightarrow F\otimes_R C\rightarrow \Tr M\otimes_RC\rightarrow0.
\end{equation}

(i)$\Rightarrow$(ii) By (\ref{def2}.1), $\Tor_1^R(\Tr M,C)=0$ and so $\lambda M\otimes_RC$ is a first $C$-syzygy module by the exact sequence (\ref{l6}.1.).
Hence $\lambda M\otimes_RC$ satisfies $\widetilde{S}_1$ by Theorem \ref{t1}.

(ii)$\Rightarrow$(i) Assume, contrarily, that $\Tor_1^R(\Tr M,C)\cong\underline{\Hom}_R(M,C)\neq0$ and that $\fp\in\Ass_R(\Tor_1^R(\Tr M,C))$. By the exact sequence (\ref{l6}.1), $\depth_{R_\fp}((\lambda M\otimes_R C)_\fp)=0$. As $\lambda M\otimes_RC$ satisfies $\widetilde{S}_1$, $\fp\in X^0$. By \cite[Theorem 2.8]{AB}, there exists the following exact sequence:
\begin{equation}\tag{\ref{l6}.2}
0\rightarrow\Ext^1_R(\mathcal{T}_2(\Tr M),C)\rightarrow\Tor_1^R(\Tr M,C)\rightarrow\Hom_R(\Ext^1_R(\Tr M,R),C).
\end{equation}
As $M$ is horizontally linked, $\Ext^1_R(\Tr M,R)=0$ by Theorem \ref{MS}. As $\fp\in X^0$, $C_\fp$ is injective. By localizing the exact sequence (\ref{l6}.2) at $\fp$, we conclude that $\Tor_1^R(\Tr M,C)_\fp\cong\Ext^1_R(\mathcal{T}_2(\Tr M),C)_\fp=0$ which is a contradiction.
\end{proof}


\begin{thm}\label{theorem3} Let $R$ be a semiperfect ring and let $I$ , $J$ be ideals of $R$ which are linked by zero ideal. Assume that $\id_{R_\fp}(C_\fp)<\infty$ for all $\fp\in X^0(R)$ and that $\frac{C}{JC}$ satisfies $\widetilde{S}_1$. The following statements hold true:
\begin{enumerate}[(i)]
\item{$\gkd_R(R/I)=0$ if and only if $\Ext^i_R(R/I,C)=0=\Ext^i_{\mathcal{I}_C}(R/J,R)$ for all $i>0$.}
\item{If $\gkd_R(R/I)<\infty$ then $R/I$ satisfies $\widetilde{S}_n$ if and only if $\Ext^i_{\Ic}(R/J,R)=0$ for all $i$, $0<i<n$.}
\item{If $(R,\fm)$ is a local ring of depth $d\geq2$ and $\gcpd((R/I)_\fp)=0$ for all $\fp\in\Spec(R)-\{\fm\}$, then $$\Ext^i_{\mathcal{I}_C}(R/J,R)\cong\hh^i_{\fm}(R/I)$$ for all $i$, $0<i<d$.}
\end{enumerate}
\end{thm}
\begin{proof}
This is an immediate consequence of Lemma \ref{l6}, Proposition \ref{t3} and Theorem \ref{t4}.
\end{proof}
Let $R$ be a Cohen-Macaulay local ring with canonical module
$\omega_R$. If $R$ is generically Gorenstein, then $\omega_R$ can be
identified with an ideal of $R$. For any such identification
$\omega_R$ is an ideal of height one or equals $R$ (see
\cite[Proposition 3.3.18]{BH}).
\begin{cor}\label{cor8}
Let $R$ be a Cohen-Macaulay local ring of dimension $d$ with canonical module $\omega_R$. Assume that the
ideals $I$ and $J$ are linked by zero ideal such that $J\omega_R=J\cap\omega_R$. If $R$ is generically Gorenstein, then the following statements hold true:
\begin{enumerate}[(i)]
\item{$R/I$ is Cohen-Macaulay if and only if $\Ext^i_{\mathcal{I}_{\omega_R}}(R/J,R)=0$ for all $i$, $1\leq i\leq d$.}
\item{$R/I$ satisfies $\widetilde{S}_n$ if and only if $\Ext^i_{\mathcal{I}_{\omega_R}}(R/J,R)=0$ for all $i$, $0<i<n$.}
\item{If $d\geq2$ and $R/I$ is generalized Cohen-Macaulay, then $$\Ext^i_{\mathcal{I}_{\omega_R}}(R/J,R)\cong\hh^i_{\fm}(R/I)$$ for all $i$, $0<i<d$.}
\end{enumerate}
\end{cor}
\begin{proof}
If $R$ is Gorenstein, then the assertion follows from \cite[Theorem 4.1]{Sc} and \cite[Corollary 3.3]{Sc}. Now assume that $R$ is not Gorenstein.
As $R$ is generically Gorenstein and it is not Gorenstein,
$\omega_R$ can be identified with an ideal of height one. The exact
sequence $0\rightarrow\omega_R\rightarrow R\rightarrow
R/\omega_R\rightarrow0$ implies the exact sequence
$$0\rightarrow\Tor_1^R(R/J,R/\omega_R)\rightarrow R/J\underset{R}{\otimes}\omega_R\rightarrow R/J\rightarrow R/J\underset{R}{\otimes} R/\omega_R\rightarrow0.$$
As $\Tor_1^R(R/J,R/\omega_R)\cong\frac{J\cap\omega_R}{J\omega_R}=0$, we obtain the following exact sequence
\begin{equation}\tag{\ref{cor8}.1}
0\rightarrow R/J\underset{R}{\otimes}\omega_R\rightarrow R/J\rightarrow
R/J\underset{R}{\otimes} R/\omega_R\rightarrow0.
\end{equation}
As $J$ is linked by $I$, $R/J$ is a first syzygy module. It follows from the exact sequence (\ref{cor8}.1) that $R/J\underset{R}{\otimes}\omega_R$ satisfies $\widetilde{S}_1$ and so the assertion is clear by Theorem \ref{theorem3}.
\end{proof}
Recall that, for an $R$--module $M$ we always have
$\gr_R(M)\leq\gkd_R(M)$. The module $M$ is called
$\gc$-\emph{perfect} if $\gr_R(M)=\gkd_R(M)$. An $R$--module $M$ is
called $\gc$-\emph{Gorenstein} if it is $\gc$-perfect and
$\Ext^n_R(M,C)$ is cyclic, where $n=\gkd_R(M)$. An ideal $I$ is
called $\gc$-perfect(resp.$\gc$-Gorenstein) if $R/I$ is
$\gc$-perfect(resp.$\gc$-Gorenstein) as $R$--module. Note that if
$I$ is a $\gc$-Gorenstein
ideal of $\gc$-dimension $n$, then $\Ext^n_R(R/I,C)\cong R/I$ (see
\cite[10]{G}).

We frequently use the following results of Golod.
\begin{lem}\label{G1} \cite[Corollary]{G}
Let $R$ be a local ring and $I$ an ideal of $R$. Assume that $K$ is an $R$--module and that $n$ is a fixed integer. If \emph{$\Ext^j_R(R/I,K)=0$} for
all $j\neq n$ then there is an isomorphism of functors \emph{$\Ext^i_{R/I}(-,\Ext^n_R(R/I,K))
\cong\Ext^{n+i}_R(-,K)$} on the category of $R/I$--modules for all $i\geq0$.
\end{lem}

\begin{thm}\label{G2} \cite[Proposition 5]{G}
Let $R$ be a local ring, $I$ a $\gc$-perfect ideal. Set $K=\Ext^{\tiny{\gr(I)}}_R(R/I,C)$.
Then the following statements hold true.
\begin{itemize}
        \item[(i)]{$K$ is a semidualizing $R/I$--module.}
         \item[(ii)]If $M$ is a $R/I$--module, then $\gkd_R(M)<\infty$ if and only if $\gkkd_{R/I}(M)<\infty$, and
         if these dimensions are finite then $\gkd_R(M)=\gr(I)+\gkkd_{R/I}(M)$.
\end{itemize}
\end{thm}
An $R$--module $M$ is said to be \emph{linked} to an $R$--module
$N$, by an ideal $\fc$ of $R$, if
$\fc\subseteq\ann_R(M)\cap\ann_R(N)$ and $M$ and $N$ are
horizontally linked as $R/{\fc}$--modules. In this situation we
denote $M\underset{\fc}{\thicksim}N$ \cite[Definition 4]{MS}.
The following is a generalization of \cite[Lemma 5.8]{DS}.
\begin{lem}\label{l1}
Let $R$ be a semiperfect ring, $\fa$ a $\gc$-perfect ideal of $R$. For a horizontally linked $R/\fa$--module $M$,
$\gr_R(M)=\gr(\fa)$.
\end{lem}
\begin{proof}
Set $n=\gr(\fa)$ and $K=\Ext^n_R(R/\fa,C)$. As $\fa\subseteq\ann_R(M)$, $\gr_R(M)\geq n$. By Lemma \ref{G1}, $\Ext^n_R(M,C)\cong\Hom_{R/\fa}(M,K)$. If $\Hom_{R/\fa}(M,K)=0$, then there exists an $K$--regular element $x$ in $\ann_{R/\fa}(M)$ by \cite[Proposition 1.2.3]{BH}. By Theorem \ref{G2}, $K$ is a semidualizing $R/\fa$--module and so $\Ass_{R/\fa}(K)=\Ass_{R/\fa}(R/\fa)$. Hence, $x$ is also an $R/\fa$-regular element and so $\Hom_{R/\fa}(M,R/\fa)=0$ by \cite[Proposition 1.2.3]{BH}. It follows from the exact sequence (\ref{d1}.1) that $\lambda_{R/\fa}M$ is free $R/\fa$--module, which is a contradiction. Therefore, $\Ext^n_R(M,C)\neq0$. As $C$ is a semidualizing $R$--module, every $R$-regular
sequence is also $C$-regular sequence.
Therefore, by \cite[Proposition 1.2.10]{BH}, we have
\[\begin{array}{rl}
\gr_R(M)&=\gr(\ann_R(M),R)\\
&\leq\gr(\ann_R(M),C)\\
&=\inf\{i\geq0\mid\Ext^i_R(M,C)\neq0\}\\
&\leq n.
\end{array}\]
\end{proof}
\begin{thm}\label{t6}
Let $(R,\fm)$ be a local ring, $\fa$ a $\gc$-perfect ideal of $R$ and  $M$, $N$ two $R/\fa$--modules. Set $K=\Ext^{\tiny{\gr(\fa)}}_R(R/\fa,C)$. Assume that $M\underset{\fa}{\thicksim}N$ and that $\underline{\Hom}_{R/\fa}(M,K)=0$. The following statements hold true.
\begin{enumerate}[(i)]
\item{$M$ is $\gc$-perfect if and only if $\Ext^i_{\mathcal{I}_K}(N,R/\fa)=0=\Ext^i_{R/\fa}(N,K)$ for all $i\geq1$;}
\item{If $\gkd_R(M)<\infty$, then $M$ satisfies $\widetilde{S}_n$ if and only if $\Ext^i_{\mathcal{I}_K}(N,R/\fa)=0$ for all $i$, $0<i<n$;}
\item{If $\depth R/\fa\geq 2$ and $M_\fp$ is $\gcp$-perfect for all $\fp\in\Spec(R)-\{\fm\}$, then
$$\Ext^i_{\mathcal{I}_K}(N,R/\fa)\cong\hh^i_\fm(M) \text{ for all } i, 0<i<\depth R/\fa.$$
In particular, if $0<\gkkd_{R/\fa}(M)<\infty$ then $\rgr_{\mathcal{I}_K}(N)=\depth_R(M)$.}
\end{enumerate}
\end{thm}
\begin{proof}
By Theorem \ref{G2}, $K$ is a semidualizing $R/\fa$-module and $\gkkd_{R/\fa}(M)<\infty$ if and only if $\gkd_R(M)<\infty$. Note that $M$ is $\gc$-perfect $R$--module if and only if $\gkkd_{R/\fa}(M)=0$ by Lemma \ref{l1} and Theorem \ref{G2}(ii). Now the assertion is clear by Proposition \ref{t3} and Theorem \ref{t4}.
\end{proof}
We are now ready to prove Theorem \ref{th1}, which is stated in the introduction.
\begin{proof}[Proof of Theorem \ref{th1}] This is clear by Lemma \ref{l6} and Theorem \ref{t6}.
\end{proof}
\section{Auslander class}
In this section, we study the theory of linkage for modules in the Auslander class.
For two $R$--modules $M$ and $N$, there exists
the following exact sequence
\begin{equation}\label{1.2}
0\longrightarrow \Ext^1_R(\Tr M,N)\longrightarrow
M\underset{R}{\otimes}N\overset{e_M^N}{\longrightarrow}\Hom_R(M^*,N)
\longrightarrow\Ext^2_R(\Tr M,N)\longrightarrow0,
\end{equation}
where $e_M^N:M\underset{R}{\otimes}N\rightarrow \Hom_R(M^*,N)$ is
the evaluation map \cite[Proposition 2.6]{AB}.
\begin{thm}\label{th5}
Let $M$ be an $R$--module.
Assume that $M\in\mathcal{A}_C$ and that $n$ is a positive integer.
Consider the following statements.
\begin{enumerate}[(i)]
\item{$\Ext^i_R(\Tr M,C)=0$ for all $i$, $1\leq i\leq n$;}
\item{$\Ext^i_R(\Tr M,R)=0$ for all $i$, $1\leq i\leq n$;}
\item{$M$ is an $n$-th syzygy $R$--module.}
\end{enumerate}
Then we have the following
\begin{enumerate}[(a)]
\item{(i)$\Leftrightarrow$ (ii) $\Rightarrow$ (iii);}
\item{ If $\gd_{R_\fp}(M_\fp)<\infty$ for all
$\fp\in\X^{n-2}(R)$(e.g. $\id_{R_\fp}(C_\fp)<\infty$ for all
$\fp\in\X^{n-2}(R)$), then all the statements (i)-(iii) are
equivalent.}
\end{enumerate}
\end{thm}
\begin{proof}
(a).(i)$\Rightarrow$(ii) We argue by induction on $n$. Let $\Ext^1_R(\Tr M,C)=0$. By the exact sequence (\ref{1.2}), we get the following exact sequence
$$0\rightarrow M\otimes_RC\rightarrow\overset{m}{\oplus}C.$$ By applying the functor $\Hom_R(C,-)$ on the above exact sequence, we get the following exact sequence, $0\rightarrow\Hom_R(C,M\otimes_RC)\rightarrow\Hom_R(C,\overset{m}{\oplus}C)$. As $M\in\mathcal{A}_C$, $M\cong\Hom_R(C,M\otimes_R C)$ and so $M$ is a first syzygy $R$--module. Now let $n>1$. By induction hypothesis $\Ext^i_R(\Tr M,R)=0$ for all $i$, $1\leq i\leq n-1$. Consider the following universal pushforward of $M$,
$$0\rightarrow M\rightarrow F\rightarrow N\rightarrow0,$$
where $F$ is free and $\Ext^1_R(N,R)=0$. By \cite[Lemma 3.9]{AB}, we obtain the following exact sequence
$$0\rightarrow\Tr N\rightarrow P\rightarrow\Tr M\rightarrow0,$$
where $P$ is a projective $R$--module. Therefore, $\Ext^i_R(\Tr N,C)=\Ext^{i+1}_R(\Tr M,C)=0$ for all $i$, $1\leq i\leq n-1$. By induction hypothesis $\Ext^{i}_R(\Tr N,R)=0$ for all $i$, $1\leq i\leq n-1$ and so $\Ext^i_R(\Tr M,R)=0$ for all $i$, $1\leq i\leq n$.

(ii)$\Rightarrow$(i). This follows from \cite[Theorem 2.12]{DS1}.

(ii)$\Rightarrow$(iii). This follows from \cite[Proposition 11]{M1}.

(b) Follows from \cite[Theorem 43]{M1}.
\end{proof}
\begin{cor}\label{cor7}
Let $R$ be a semiperfect ring, $M$ a stable $R$--module. Assume that $M\in\mathcal{A}_C$ and that
$n$ is a positive integer. If $\gd_{R_\fp}(M_\fp)<\infty$ for all
$\fp\in\X^{n-2}(R)$(e.g. $\id_{R_\fp}(C_\fp)<\infty$ for all
$\fp\in\X^{n-2}(R)$), then the following are equivalent.
\begin{enumerate}[(i)]
\item{$M$ is an $n$-th syzygy module;}
\item{$M$ is horizontally linked and $\Ext^i_R(\lambda M,C)=0$ for all $i$,
$0<i<n$.}
\end{enumerate}
\end{cor}
\begin{proof}
This is clear by Theorem \ref{th5} and Theorem \ref{MS}.
\end{proof}
Note that every module of finite projective dimension is in the Auslander class with respect to $C$. As a consequence we have the following result.
\begin{cor}
Let $R$ be a semiperfect ring, $M$ a stable $R$--module of finite projective dimension and $n$ a positive integer. The following are equivalent:
\begin{enumerate}[(i)]
\item{$M$ satisfies $\widetilde{S}_n$;}
\item{$M$ is an $n$-th syzygy module;}
\item{$M$ is horizontally linked and $\Ext^i_R(\lambda M,C)=0$ for all $i$, $0<i<n$ and some semidualizing module $C$.}
\end{enumerate}
\end{cor}
\begin{proof}
This is clear by Theorem \ref{th5}, Theorem \ref{MS} and \cite[Theorem 2.12]{DS1}.
\end{proof}
Let $R$ be a Cohen-Macaulay local ring with canonical module $\omega_R$. It is well-known that $\Hom_R(C,\omega_R)$
is a semidualizing $R$--module.
\begin{cor}
Let $R$ be a Cohen-Macaulay local ring with canonical module $\omega_R$, $M$ a stable $R$--module. Set $K:=\Hom_R(C,\omega_R)$. If $\gkd_R(M)<\infty$ and $\id_{R_\fp}(K_\fp)<\infty$ for all $\fp\in X^{n-1}$, then the following are equivalent.
\begin{enumerate}[(i)]
\item{$M$ satisfies $\widetilde{S}_n$;}
\item{$M$ is an $n$-th syzygy module;}
\item{$M$ is horizontally linked and $\Ext^i_R(\lambda M,K)=0$ for all $i$, $0<i<n$.}
\end{enumerate}
\end{cor}
\begin{proof}
This is an immediate consequence of Corollary \ref{cor7}, Example \ref{example1}(iv) and \cite[Corollary 2.14]{DS1}.
\end{proof}
\begin{thm}\label{t7}
Let $R$ be a semiperfect ring, $M$ a horizontally linked $R$--module of finite $\gc$-dimension and $\lambda M\in\ac$(e.g. $\Ic$-$\id_R(\lambda M)<\infty$). Then the following are equivalent:
\begin{enumerate}[(i)]
\item{$\gd_R(M)=0$;}
\item{$\gd_R(\lambda M)=0$;}
\item{$\lambda M$ satisfies $\widetilde{S}_n$ for some $n>\depth R-\depth_R(M)$;}
\item{$\gkd_R(M)=0$.}
\end{enumerate}
\end{thm}
\begin{proof}
The equivalence of $(i)$ and $(ii)$ follows from \cite[Theorem 1]{MS}.

(ii)$\Rightarrow$(iii). This is clear by Theorem \ref{G3}.

(iii)$\Rightarrow$(iv). By \cite[Theorem 4.6]{DS1}, $\rgr(M,C)\geq n$. Now
assume contrarily that $\gkd_R(M)>0$. Hence we obtain the following inequality from Theorem
\ref{G3}:
$$n\leq\rgr(M,C)\leq\gkd_R(M)=\depth R-\depth_R(M),$$
which is a contradiction.

(iv)$\Rightarrow$(i). By Example \ref{example1}(i), $\Tr M\in\ac$ and so $\Ext^i_R(\Tr M,R)\cong\Ext^i_R(\trk M,C)=0$ for all $i>0$ by Remark \ref{rem2}(ii) and Theorem \ref{G3}(i).
As $\gkd_R(M)=0$, $\Ext^i_R(M,C)=0$ for all $i>0$ by Theorem \ref{G3}. It follows from Theorem \ref{th5} that $\Ext^i_R(M,R)=0$ for all $i>0$. Hence $\gd_R(M)=0$.
\end{proof}
\begin{cor}
Let $R$ be a local ring, $\fc$ a $\gc$-perfect ideal of $R$ and  $M$, $N$ two $R/\fc$--modules. Set $K=\Ext^{\gr(\fc)}_R(R/\fc,C)$. Assume that $M\underset{\fc}{\thicksim}N$ and that $N\in\mathcal{A}_K$. If $\gkd_R(M)<\infty$, then the following statements hold true.
\begin{enumerate}[(i)]
\item{$\gd_{R/\fc}(M)=0$;}
\item{$\gd_{R/\fc}(N)=0$;}
\item{$N$ satisfies $\widetilde{S}_n$ for some $n>\depth R/\fc-\depth_{R/\fc}(M)$;}
\item{$M$ is $\gc$-perfect $R$--module.}
\end{enumerate}
\end{cor}
\begin{proof}
By Theorem \ref{G2}, $K$ is a semidualizing $R/\fa$-module and $\gkkd_{R/\fa}(M)<\infty$. Note that $M$ is $\gc$-perfect $R$--module if and only if $\gkkd_{R/\fa}(M)=0$ by Lemma \ref{l1} and Theorem \ref{G2}(ii). Now the assertion is clear by Theorem \ref{t7}.
\end{proof}
\section{Serre condition and vanishing of local cohomology}
In this section we study the connection of the Serre condition $(S_n)$ on a horizontally linked $R$--module of finite $\gc$-dimension with the vanishing of the local cohomology groups of its linked module. First we bring the following Lemma which is a generalization of \cite[Proposition 2.6]{AB}.
\begin{lem}\label{l5}
Let $M$ and $N$ be $R$--modules. If $N\in\bc$, then there is a natural exact sequence
$$0\rightarrow\Ext^1_R(\trk M,N)\rightarrow M\otimes_R\Hom_R(C,N)\rightarrow\Hom_R(M^\triangledown,N)\rightarrow\Ext^2_R(\trk M,N)\rightarrow0.$$
\end{lem}
\begin{proof}
Let $P_1\overset{g}{\rightarrow} P_0\overset{f}{\rightarrow}M\rightarrow0$ be a projective presentation of $M$. By applying the functor $(-)^{\triangledown}$ on the projective presentation of $M$, we obtain the following exact sequences:
\begin{equation}\tag{\ref{l5}.1}
0\longrightarrow M^{\triangledown}\overset{f^{\triangledown}}{\longrightarrow} (P_0)^{\triangledown}\overset{{\alpha}_1}\longrightarrow L\longrightarrow0,
\end{equation}
\begin{equation}\tag{\ref{l5}.2}
0\longrightarrow L\overset{{\alpha}_2}{\longrightarrow} (P_1)^{\triangledown}\longrightarrow \trk M\longrightarrow0,
\end{equation}
where $L=\coker(f^{\triangledown})$ and $\alpha_2\circ\alpha_1=g^{\triangledown}$.
As $N\in\bc$, $\Ext^i_R(C,N)=0$ for all $i>0$.
By applying the functor $(-)^{\dagger}:=\Hom_R(-,N)$ on the exact sequences (\ref{l5}.1) and (\ref{l5}.2), we obtain the following exact sequences:
\begin{equation}\tag{\ref{l5}.3}
0\rightarrow\Hom_R(L,N)\overset{{\alpha_1}^{\dagger}}{\rightarrow}\Hom_R((P_0)^{\triangledown},N)\overset{{f^{\triangledown}}^{\dagger}}{\rightarrow}\Hom_R(M^{\triangledown},N)\rightarrow\Ext^1_R(L,N)\rightarrow0
\end{equation}
\begin{equation}\tag{\ref{l5}.4}
0\rightarrow\Hom_R(\trk M,N)\rightarrow\Hom_R((P_1)^{\triangledown},N)\overset{{\alpha_2}^{\dagger}}{\rightarrow}\Hom_R(L,N)\rightarrow\Ext^1_R(\trk M,N)\rightarrow0
\end{equation}
Note that the homomorphism evaluation $\theta_i:P_i\otimes_R\Hom_R(C,N)\rightarrow\Hom_R((P_i)^{\triangledown},N)$ is an isomorphism for $i=0,1$.
Consider the following commutative diagram with exact rows:
$$\begin{CD}
P_1\otimes_R\Hom_R(C,N)@>\varphi>>P_0\otimes_R\Hom_R(C,N)@>\psi>>M\otimes_R\Hom_R(C,N)\rightarrow0  \\
@VV\beta V@VV\theta_0V@VV\theta V\\
0\rightarrow\Hom_R(L,N)@>(\alpha_1)^{\dagger}>>\Hom_R((P_0)^{\triangledown},N)@>{f^{\triangledown}}^{\dagger}>>\Hom_R((M)^{\triangledown},N)\\
\end{CD}
$$
where $\beta={\alpha_2}^{\dagger}\circ\theta_1$. By the exact sequence (\ref{l5}.4), $\coker(\beta)=\coker({\alpha_2}^{\dagger})\cong\Ext^1_R(\trk M,N)$. Now the Snake Lemma implies that $\ker(\theta)\cong\coker(\beta)\cong\Ext^1_R(\trk M,N)$.

Since $\psi$ is surjective and $\theta_0$ is an isomorphism, we get $\im({f^{\triangledown}}^{\dagger})=\im(\theta)$ and so $\coker(\theta)=\coker({f^{\triangledown}}^{\dagger})\cong\Ext^1_R(L,N)$ by the exact sequence (\ref{l5}.3).
By applying the functor $\Hom_R(-,N)$ on the exact sequence (\ref{l5}.2), we obtain the isomorphism $\Ext^1_R(L,N)\cong\Ext^2_R(\trk M,N)$.
\end{proof}
\begin{thm}\label{t5}
Let $R$ be a Cohen-Macaulay local ring with canonical module $\omega_R$.
For an $R$--module $M$ of finite $\gc$-dimension, the following are equivalent:
\begin{enumerate}[(i)]
\item{$M$ satisfies $\widetilde{S}_n$;}
\item{$M$ is an $n$th $\omega_R$-syzygy module;}
\item{$M$ is an $n$th $C$-syzygy module;}
\item{$\Ext^i_R(\trk M,\omega_R)=0$ for all $i$, $1\leq i\leq n$;}
\item{$\Ext^i_R(\trk M,C)=0$ for all $i$, $1\leq i\leq n$;}
\item{$\Ext^i_R(\tro M,\omega_R)=0$ for all $i$, $1\leq i\leq n$.}
\end{enumerate}
\end{thm}
\begin{proof}
The equivalence of (i), (iii) and (v) follows from Theorem \ref{t1}.

(iv)$\Rightarrow$(i). First we prove that $M\otimes_R\Hom_R(C,\omega_R)$ is an $n$th $\omega_R$-syzygy module. Applying $(-)^{\dag}:=\Hom_R(-,\omega_R)$
to a projective resolution $\cdots\rightarrow
P_{n-1}\rightarrow\cdots\rightarrow P_0\rightarrow
M^{\triangledown}\rightarrow0$ of $M^{\triangledown}=\Hom_R(M,C)$, implies the following
complex
\begin{equation}\tag{\ref{t5}.1}
0\rightarrow (M^{\triangledown})^{\dag}\rightarrow (P_0)^{\dag}\rightarrow\cdots\rightarrow
(P_{n-2})^{\dag}\rightarrow(P_{n-1})^{\dag}.
\end{equation}
By Example \ref{example1}(i), $\omega_R\in\bc$. Note by Lemma \ref{l5}
that $M\otimes_RC^{\dag}$ is embedded in
$(M^{\triangledown})^{\dag}$ if $n=1$ and $M\otimes_RC^{\dag}\cong
(M^{\triangledown})^{\dag}$ if $n>1$. Therefore $M\otimes_RC^{\dag}$ is
$1$st $\omega_R$--syzygy if $n=1$ and, $M\otimes_RC^{\dag}$ is $2$nd $\omega_R$-syzygy for $n>1$. Assume that
$n>2$. As $\Ext^i_R(C,\omega_R)=0$ for all $i>0$,  $\Ext^i_R(\trk
M,\omega_R)\cong\Ext^{i-2}_R(M^\triangledown,\omega_R)$ for all $i>2$,
by the exact sequence (\ref{d3}.1). Therefore the complex (\ref{t5}.1) is
exact, i.e. $M\otimes_RC^{\dag}$ is an $n$th $\omega_R$-syzygy. By Theorem \ref{t1}, $M\otimes_RC^{\dag}$ satisfies $\widetilde{S}_n$. As $\gkd_R(M)<\infty$, $M\in\mathcal{A}_{C^{\dag}}$ by Example \ref{example1}(iv). Hence $M$ satisfies $\widetilde{S}_n$ by Lemma \ref{lem2}.

(i)$\Rightarrow$(iv).
We argue by induction on $n$. If $n=1$ then it is enough to
show that $\Ass_R(\Ext^1_R(\trk M,\omega_R))=\emptyset$. Assume contrarily
that $\fp\in\Ass_R(\Ext^1_R(\trk M,\omega_R))$. By Lemma \ref{l5}, $\depth_{R_\fp}(M\otimes_RC^{\dagger})_\fp=0$. As $\gkd_R(M)<\infty$, $M\in\mathcal{A}_{C^{\dagger}}$ by Example \ref{example1}(iv) and so $\depth_{R_\fp}(M_\fp)=0$ by Lemma \ref{lem2}. As $M$ satisfies $\widetilde{S}_1$, we conclude that $\fp\in X^{0}(R)$. Hence $\Ext^1_R(\trk
M,\omega_R)_\fp=0$, which is a contradiction.

Now, let $n>1$. By induction hypothesis, $\Ext^i_R(\trk M,\omega_R)=0$ for all $i$, $1\leq i\leq n-1$. By (v), $\Ext^1_R(\trk M, C)=0$. As noted, after Lemma \ref{l4},
there is a universal
pushforward
\begin{equation}\tag{\ref{t5}.1}
0\longrightarrow M\longrightarrow C^m\longrightarrow
N\longrightarrow0.
\end{equation}
of $M$ with respect to $C$. As $\Ext^1_R(N,C)=0$, the exact sequence
(\ref{t5}.1) implies the exact sequence
\begin{equation}\tag{\ref{t5}.2}
0\longrightarrow\trk N\longrightarrow C^m\longrightarrow\trk M\longrightarrow0,
\end{equation}
by \cite[Lemma 2.2]{DS1}. Note that $N$ has finite
$\gc$--dimension. As $M$ satisfies $\widetilde{S}_n$
and $\Ext^1_R(N,C)=0$, it is easy to see that $N$ satisfies
$\widetilde{S}_{n-1}$. Induction hypothesis on $N$, gives that
$\Ext^i_R(\trk N,\omega_R)=0$ for all $i$, $1\leq i\leq n-1$. Finally, the
exact sequence (\ref{t5}.2) implies that $\Ext^i_R(\trk M,\omega_R)=0$ for
all $i$, $1\leq i\leq n$.

Note that $\gomegad_R(M)<\infty$. Hence, the equivalence of (i), (ii) and (vi) follows from Theorem \ref{t1}.
\end{proof}
As a consequence, we can generalize \cite[Theorem 4.2]{DS} as follows.
\begin{cor}\label{cor2}
Let $(R,\fm)$ be a Cohen-Macaulay local ring of dimension $d$. Assume that $M$ is a stable $R$--module of finite $\gc$-dimension such that $\underline{\Hom}_R(M,C)=0$, then the following are equivalent:
\begin{enumerate}[(i)]
\item{$M$ satisfies $\widetilde{S}_n$;}
\item{$M$ is horizontally linked and $\hh^{i}_{\fm}(\lambda M\otimes_RC)=0$ for all $i$, $d-n<i<d$.}
\end{enumerate}
\end{cor}
\begin{proof}
We may assume that $R$ is complete with canonical module $\omega_R$. By Example \ref{example1}(i), $\omega_R\in\bc$ and so $\Ext^i_R(C,\omega_R)=0$ for all $i>0$. By \cite[Theorem 10.62]{R}, there is the following spectral sequence:
$$\E^{p,q}_2=\Ext^p_R(\Tor_q^R(\Tr M,C),\omega_R)\Rightarrow\Ext^n_R(\Tr M,\Hom_R(C,\omega_R))$$
and so we get the following exact sequence:
\begin{equation}\tag{\ref{cor2}.1}
0\rightarrow\Ext^1_R(\Tr M\otimes_RC,\omega_R)\rightarrow\Ext^1_R(\Tr M,\Hom_R(C,\omega_R))\rightarrow\Hom_R(\Tor_1^R(\Tr M,C),\omega_R),
\end{equation}
by \cite[Theorem 10.33]{R}. As $\Tor_1^R(\Tr M,C)\cong\underline{\Hom}_R(M,C)=0$, we obtain the following exact sequence
\[\begin{array}{rl}\tag{\ref{cor2}.2}
\Ext^1_R(\trk M,\omega_R)&\cong\Ext^1_R(\Tr M\otimes_RC,\omega_R)\\
&\cong\Ext^1_R(\Tr M,\Hom_R(C,\omega_R)),
\end{array}\]
by (\ref{cor2}.1) and Remark \ref{remark3}(i).
By theorem \ref{MS}, $M$ is horizontally linked if and only if $\Ext^1_R(\Tr M,R)=0$. As $\gkd_R(M)<\infty$, $M\in\mathcal{A}_{C^{\dag}}$ by Example \ref{example1}(iv). By Theorem \ref{th5},
$\Ext^1_R(\Tr M,R)=0$ if and only if $\Ext^1_R(\Tr M, \Hom_R(C,\omega_R))=0$ and this is equivalent to say that $\Ext^1_R(\trk M,\omega_R)=0$ by (\ref{cor2}.2).  By Remark \ref{remark3}(ii), we have the following isomorphism:
\begin{equation}\tag{\ref{cor2}.3}
\Ext^i_R(\lambda M\otimes_RC,\omega_R)\cong\Ext^{i+1}_R(\trk M,\omega_R) \text{ for all } i>0.
\end{equation}
By Theorem \ref{t5}, $M$ satisfies $\widetilde{S}_n$ if and only if $\Ext^i_R(\trk M,\omega_R)=0$ for all $i$, $1\leq i\leq n$ and this is equivalent to say that $M$ is horizontally linked and $\Ext^i_R(\lambda M\otimes_RC,\omega_R)=0$ for all $i$, $1\leq i\leq n-1$ by (\ref{cor2}.2) and (\ref{cor2}.3). Now the assertion is clear by the Local Duality Theorem \cite[Corollary 3.5.9]{BH}.
\end{proof}
As an immediate consequence, we have the following results.
\begin{cor}\label{cor4}
Let $(R,\fm)$ be a Cohen-Macaulay local ring, $\fa$ a $\gc$-perfect ideal of $R$. Set $K=\Ext^{\tiny{\gr(\fa)}}_R(R/\fa,C)$. Assume that $M$ and $N$ are $R/\fa$--modules such that $M\underset{\fa}{\thicksim}N$ and $\underline{\Hom}_{R/\fa}(M,K)=0$. If $\gkd_R(M)<\infty$ then
$M$ satisfies $\widetilde{S}_n$ if and only if $\hh_{\fm}^i(N\otimes_RK)=0$ for all $i$, $\dim(R/\fa)-n<i<\dim(R/\fa)$.
\end{cor}
\begin{proof}
This is clear by Corollary \ref{cor2} and Theorem \ref{G2}.
\end{proof}
Note that, for a horizontally linked module $M$ over a Cohen-Macaulay local ring $R$, the properties $\widetilde{S}_k$ and $(S_k)$  are identical.
\begin{cor}\label{cor5} Let $(R,\fm)$ be a Cohen-Macaulay local ring and let $I$ and $J$ be two ideals of $R$ linked by the zero ideal. Assume that $\id_{R_\fp}(C_\fp)<\infty$ for all $\fp\in X^0(R)$ and that $\frac{C}{JC}$ satisfies $\widetilde{S}_1$. If $\gkd_R(R/I)<\infty$, then the following are equivalent:
\begin{enumerate}[(i)]
\item{$R/I$ satisfies $(S_n)$;}
\item{$\hh^i_{\fm}(\frac{C}{JC})=0$ for all $i$, $d-n<i<d$.}
\end{enumerate}
\end{cor}
\begin{proof}
This is clear by Corollary \ref{cor2} and Lemma \ref{l6}.
\end{proof}
\begin{thm}\label{cor6}
Let $R$ be a Cohen-Macaulay local ring. Assume that $M$ is a horizontally linked $R$--module of finite $\gc$-dimension and that $\lambda M\in\mathcal{A}_C$ (e.g. $\Ic$-$\id_R(\lambda M)<\infty$). The following are equivalent:
\begin{enumerate}[(i)]
\item{$M$ is maximal Cohen-Macaulay;}
\item{$\lambda M$ is maximal Cohen-Macaulay;}
\item{$M$ satisfies $(S_n)$ for some $n>\depth R-\depth_R(\lambda M)$;}
\item{$\lambda M$ satisfies $(S_n)$ for some $n>\depth R-\depth_R(M)$.}
\end{enumerate}
\end{thm}
\begin{proof}
The equivalence of (i), (ii) and (iv) follows from \cite[Corollary 4.7]{DS1}.

(iii)$\Rightarrow$(ii). By Example \ref{example1}(i), $\Tr M\in\ac$, and so $\underline{\Hom}_R(M,C)\cong\Tor_1^R(\Tr M,C)=0$. By Corollary \ref{cor2}, $\hh^i_\fm(\lambda M\otimes_RC)=0$ for all $i$, $\depth R-n<i<\depth R$. Set $t=\depth_R(\lambda M)$. Note that $\depth_R(\lambda M\otimes_RC)=t$ by Lemma \ref{lem2}.
Hence $\hh^i_\fm(\lambda M\otimes_RC)=0$ for all $i$, $i<t$. As $t>\depth R-n$, we conclude that $\lambda M\otimes_RC$ is maximal Cohen-Macaulay. Therefore $\lambda M$ is maximal Cohen-Macaulay by using Lemma \ref{lem2} again.

(i)$\Rightarrow$(iii). If $M$ is maximal Cohen-Macaulay then $M$ satisfies $(S_n)$ for all $n$.
\end{proof}
The following is a generalization of \cite[Theorem 10]{MS}.
\begin{thm}\label{t8}
Let $(R,\fm)$ be a Cohen-Macaulay local ring of dimension $d\geq2$ and $M$ a horizontally linked $R$--module.
Assume that $\gcpd(M_\fp)<\infty$ for all $\fp\in\Spec(R)-\{\fm\}$ and that $\lambda M\in\ac$. Then the following statements hold true:
\begin{enumerate}[(i)]
\item{$M$ is generalized Cohen-Macaulay if and only if $\lambda M$ is;}
\item{If $M$ is generalized Cohen-Macaulay, then $$\hh^i_{\fm}(M)\cong\Ext^i_R(\lambda M,R) \text{ for all } i, 0<i<d,$$
$$\hh^i_{\fm}(\lambda M)\cong\Ext^i_R(M,R) \text{ for all } i, 0<i<d;$$
In particular, If $M$ is not maximal Cohen-Macaulay, then $\depth_R(M)=\rgr(\lambda M)$.}
\end{enumerate}
\end{thm}
\begin{proof}
By \cite[Lemma 1.2 , Lemma 1.4]{T}, $M$ is generalized Cohen-Macaulay module if and only if $M_\fp$ is a maximal Cohen-Macaulay module $R_\fp$--module for all $\fp\in\Spec(R)-\{\fm\}$. Hence the first assertion is clear by Corollary \ref{cor6}.

(ii) If $M$ is generalized Cohen-Macaulay, then $\gcpd(M_\fp)=0$ for all $\fp\in\Spec (R)-\{\fm\}$ by Theorem \ref{G3}(iv).
Hence $\hh^i_{\fm}(M)\cong\Ext^i_R(\lambda M,R)$ for all $i$, $0<i<d$ by Corollary \ref{cor1}(ii). By Theorem \ref{t7}, $\gd_{R_\fp}(M_\fp)=0$ for all $\fp\in\Spec (R)-\{\fm\}$ and so $\Ext^i_R(M,R)$ is finite length for all $i>0$. Now the second isomorphism follows from \cite[Theorem 4.4]{DS}.
\end{proof}
As an immediate consequence, we have the following result.
\begin{cor}\label{cor9}
Let $(R,\fm)$ be a Cohen-Macaulay local ring, $\fc$ a $\gc$-perfect ideal of $R$. Set $K=\Ext^{\gr(\fc)}_R(R/\fc,C)$. Assume that $M$, $N$ are $R/\fc$--modules such that $M\underset{\fc}{\thicksim}N$. If $\gkd_R(M)<\infty$ and $N\in\mathcal{A}_K$, then the following statements hold true:
\begin{enumerate}[(a)]
\item{The following are equivalent:
\begin{enumerate}[(i)]
\item{$M$ is Cohen-Macaulay;}
\item{$N$ is Cohen-Macaulay;}
\item{$M$ satisfies $(S_n)$ for some $n>\depth(R/\fc)-\depth_{R/\fc}(N)$;}
\item{$N$ satisfies $(S_n)$ for some $n>\depth(R/\fc)-\depth_{R/\fc}(M)$.}
\end{enumerate}}
\item{If $\dim R/\fc\geq2$, then $M$ is generalized Cohen-Macaulay if and only if $N$ is. Moreover, if $M$ is generalized Cohen-Macaulay, then
$$\hh^i_\fm(M)\cong\Ext^i_{R/\fc}(N,R/\fc) \text{ for all } i, 0<i<d.$$ In particular, if $M$ is not Cohen-Macaulay then $\depth_R(M)=\rgr_{R/\fc}(N)$.}
\end{enumerate}
\end{cor}
\begin{proof}
This is clear by Theorem \ref{cor6}, Theorem \ref{t8} and Theorem \ref{G2}.
\end{proof}
Note that if $\fc$ is a $\gc$-Gorenstein ideal then $K=\Ext^{\gr(\fc)}_R(R/\fc,C)\cong R/\fc$. Hence Theorem \ref{th2} is a special case of Corollary\ref{cor9}.
\bibliographystyle{amsplain}

\end{document}